\documentclass[a4paper,fleqn]{cas-dc}

\usepackage{geometry}
\usepackage{CJK}                		

\usepackage{todonotes}
\usepackage{graphicx}
\usepackage{subfigure}					
\usepackage{amssymb}
\usepackage{amsmath}
\usepackage{amsthm}
\usepackage{float}
\usepackage{verbatim}
\usepackage{listings}
\usepackage[linesnumbered,boxed,ruled,commentsnumbered]{algorithm2e}
\newtheorem{theorem}{Theorem}
\newtheorem{definition}{Definition}
\newtheorem{lemma}{Lemma}
\newtheorem{proposition}{Proposition}
\newtheorem{assumption}{Assumption}
\newtheorem{remark}{Remark}

\newtheorem{example}{Example}
\newtheorem{corollary}{Corollary}
\newcommand{\interior}[1]{%
  {\kern0pt#1}^{\mathrm{o}}%
}
\newcommand{\bbR}{\mathbb{R}}

\newcommand\ignore[1]{}

\newcommand\edits[1]{{#1}}

\usepackage{amssymb}
\usepackage[authoryear,longnamesfirst]{natbib}

\def\tsc#1{\csdef{#1}{\textsc{\lowercase{#1}}\xspace}}
\tsc{WGM}
\tsc{QE}
\tsc{EP}
\tsc{PMS}
\tsc{BEC}
\tsc{DE}

\begin{document}
\let\WriteBookmarks\relax
\def\floatpagepagefraction{1}
\def\textpagefraction{.001}

\shorttitle{On the Softplus Penalty for Constrained Convex Optimization}

\shortauthors{Li, Grigas, and Atamtürk}

\title [mode = title]{On the Softplus Penalty for Constrained Convex Optimization}                      

\tnotetext[0]{This research is supported, in part, by NSF AI Institute for Advances in Optimization Award 2112533, and DOD ONR grant 12951270.
}

\tnotemark[1]

\tnotetext[1]{Corresponding author
}

\author[1]{Meng Li}
\ead{meng_li@berkeley.edu}
\author[1]{Paul Grigas}
\cormark[1]
\ead{pgrigas@berkeley.edu}
\author[1]{Alper Atamtürk}
\ead{atamturk@berkeley.edu}

\affiliation[1]{organization={Department of Industrial Engineering and Operations Research, University of California-Berkeley},
    city={Berkeley},
    postcode={94720}, 
    state={CA},
    country={USA}}

\renewcommand{\printorcid}{}

\begin{abstract}
We study a new penalty reformulation of constrained convex optimization based on the softplus penalty function. We develop novel and tight upper bounds on the objective value gap and the violation of constraints for the solutions to the penalty reformulations by analyzing the solution path of the reformulation with respect to the smoothness parameter. We use these upper bounds to analyze the complexity of applying gradient methods, which are advantageous when the number of constraints is large, to the reformulation.
\end{abstract}


\maketitle

\section{Introduction}

Consider the convex optimization problem 
\begin{equation}\label{eq:Original Nonlinear}
\begin{aligned}
F^* := \quad \min_{x\in\mathbb{R}^n}\quad &F(x)=\frac1{\ell}\sum_{i=1}^{{\ell}}f_i(x)+\psi(x)\\
\mathrm{s.t.}\quad &a_i(x)\le 0, \  i=1,...,{m},
\end{aligned}
\end{equation}
where \(f_i: \mathbb{R}^n\to\mathbb{R}\), \(i=1, \ldots,{\ell},\) are convex, smooth functions, 
\(\psi:\mathbb{R}^n\to\mathbb{R}\) is a convex function for which one can build a proximal mapping, and $a_i: \mathbb{R}^n\to\mathbb{R}, \, i = 1, \ldots, {m},$ are convex, smooth functions. We are motivated to solve problems with numerous nonlinear constraints, which are common in practice. For example, strong convex relaxations of nonconvex optimization problems are often formulated using many SOCP or semidefinite constraints. Other examples include the minimum-volume covering ellipsoid problem \citep{sun2004computation} and quadratically-constrained quadratic programming.

In this paper, similar to \cite{li2022new}, which treats the special case with linear constraints, we construct a penalty reformulation of \eqref{eq:Original Nonlinear} utilizing the \emph{softplus} penalty function.  \cite{li2022new} establish upper bounds on the distance between the optimal solutions of the constrained problem and penalized problem, which are used to prove convergence results of the stochastic gradient algorithms for the constrained case. Their analysis for the linearly constrained case critically depends on the global upper bounds on norms of the gradients of the constraints. Unfortunately, as no such global upper bound exists for nonlinear constraints, the analysis does not extend to \eqref{eq:Original Nonlinear}.

Instead, here, we establish bounds on the objective value gap and the violations of constraints for the solutions to the penalty reformulations (Section~\ref{sec:penalty}). In particular, for the strongly convex case, we build bounds for violations of constraints by differentiating the trajectories of the optimal solutions for penalty reformulations. The major challenges of the analysis for nonlinear constraints are ($i$) the unboundedness of the norm of the gradients and ($ii$) the lack of twice continuous differentiability. We address these challenges by bounding the norms in a region close to the optimal solution based on a natural growth condition and with approximations by twice continuously differentiable functions. We present several examples to demonstrate the tightness of these bounds.
Given the bounds, we then establish the convergence rate for a simple penalty method that solves \eqref{eq:Original Nonlinear} by solving a static penalty reformulation with first-order deterministic and stochastic methods (Section \ref{sec:applications}). For example, in the $\mu$-strongly convex case we show that when applying catalyst-accelerated SVRG, we obtain a complexity proportional to 
\(\tilde O\left(m^{\frac54}/\sqrt{\mu\epsilon}\right)\) 
expected stochastic gradient iterations to obtain an \(\epsilon\)-approximate solution for a problem with $m$ constraints. We also briefly show how one can deal with non-smooth constraints by combining the approximations with our penalty method. The \emph{softplus} penalty has some other advantages. Due to the connection between \emph{softplus} penalty and Nesterov's smoothing technique \citep{nesterov2005smooth}, it allows one to obtain dual solutions when solving penalty reformulations. More advanced methods, including ones dynamically adjusting penalties through a nested structure and incorporating safe screening similar to \cite{li2022new} can also be developed based on the main results in Section~\ref{sec:penalty}.%

\edits{Various penalty methods have been considered in previous works for similar problems \citep{tatarenko2018smooth,nedic2020convergence,lan2013iteration,fercoq2019almost}, including linearly constrained problems, conically constrained problems, and problems with random linear constraints that hold almost surely.} In particular, \cite{mishchenko2018stochastic} consider a more general version of \eqref{eq:Original Nonlinear}, where the feasible region is defined as \(\cap_{i=1}^m X_i\), \(i=1,...,m\), where \(X_i\) are closed convex sets. They use the sum of squared distance functions \(\frac1{2m}\sum_{i=1}^md_{{X}_i}(x)^2\) to penalize the constraints. In Section \ref{sec:applications}, we show that compared to their results, our penalty method requires weaker assumptions, and our convergence rate has a different dependence on the problem parameters, which is better in most cases. 

This paper is organized as follows. In Section \ref{sec:penalty}, we introduce the penalty reformulation for the constrained problem \eqref{eq:Original Nonlinear} and study its properties, focusing on the estimations of violations of constraints for the solutions to the penalty reformulations. In Section \ref{sec:applications}, we apply these properties to obtain approximate solutions to \eqref{eq:Original Nonlinear} by solving penalty reformulations under different circumstances (Section~\ref{subsec:complexity}) and consider extensions to nonsmooth constrained problems (Section~\ref{subsec:approx constraints}).

\section{Penalty Function Reformulation}
\label{sec:penalty}
In this section, we present a smooth penalized reformulation of problem \eqref{eq:Original Nonlinear} and examine several of its useful theoretical properties. To penalize the constraints, we utilize the \textit{softplus} function: 
\begin{equation}\label{Softplus Penalty}
    p_{\delta}(t):=\delta\log(1+\exp(t/\delta)),\ (\delta,t)\in[0,+\infty)\times\mathbb{R},
\end{equation}
where $\delta \geq 0$ is a parameter controlling the smoothness of $p_{\delta}$ and \(p_0(t)=\max(0,t)\), by convention.
The function $p_\delta$ arises as a loss function in logistic regression, as an activation function in neural network modeling, and has been introduced as a penalty function for linearly constrained problems by \cite{li2022new}. The softplus penalty may also be viewed as an instance of Nesterov's smoothing technique~\citep{nesterov2005smooth} applied to the hinge function $\max(0,t)$, as described in \cite{li2022new}. \edits{Proposition 2.1 in \cite{li2022new} provides a list of useful properties of \(p_{\delta}\).} %

Let us now consider the penalty reformulation of \eqref{eq:Original Nonlinear}  
\begin{equation}\label{eq:Penalty Nonlinear}
    \min_{x \in \bbR^n} F_{\xi ,\delta}(x):=F(x)+\xi  \sum_{i=1}^{m} p_{\delta}(a_i(x)),
\end{equation}
where $\xi  \geq 0$ is the penalty parameter and $\delta \geq 0$ is the smoothness parameter of the softplus function. 
Let $x_{\xi ,\delta}^\ast$ denote an optimal solution of \eqref{eq:Penalty Nonlinear} for given $\xi$ and $\delta$, which is uniquely defined whenever $F$ is strongly convex.
We are interested in studying the relationship between an optimal solution $x^\ast$ for \eqref{eq:Original Nonlinear} and $x_{\xi ,\delta}^\ast$ as we vary the parameters $\xi $ and $\delta$. Towards this goal, throughout the paper, the following assumptions are made concerning problem \eqref{eq:Original Nonlinear}.
\begin{assumption}\label{assum:Assumption Nonlinear Obj}
The objective function of problem \eqref{eq:Original Nonlinear} satisfies the following properties:
\begin{enumerate}
\item $F$ is globally $\mu$-strongly convex for some $\mu \geq 0$, i.e., $F - \frac\mu2 \|\cdot\|_2^2$ is a convex function;
\item The component functions $f_i$, $i=1,...,{\ell}$, are convex and globally $L_f$-smooth for some $L_f \geq 0$, i.e.,  $\|\nabla f_i(x) - \nabla f_i(y)\|_2 \leq L_f\|x - y\|_2$ for all $x, y \in \bbR^n$;
\item The proximal term \(\psi\) is closed and convex.
\end{enumerate}
\end{assumption}
Note that we consider both the general convex case where $\mu = 0$ and the strongly convex case where $\mu > 0$.
Furthermore, we make the following assumptions on the constraints, \(a_i(x) \leq 0\), \(i=1,...,m\).
\begin{assumption}\label{assum:Assumption Nonlinear Constraints}
The constraints of problem \eqref{eq:Original Nonlinear} satisfy the following properties:
\begin{enumerate}
\item The constraint functions \(a_i\), \(i=1,...,m\), are convex and globally \(L_a\)-smooth for some \(L_a\ge0\);
\item There exist \(s>0\) and \(\hat x\in\mathbb{R}^n\) such that \(a_i(\hat x)\le -s\), \(i=1,...,m\);
\item There exist \(C_0>0\), \(C_1\ge0\),  so that \(\|\nabla a_i(x)\|_2^2\le C_0+C_1 |a_i(x)| \) for all \(x\in\bbR^{n}\), \(i=1,...,m\).
\end{enumerate}
\end{assumption}
The second assumption, the Slater condition, is made to guarantee that the KKT conditions hold at the optimum. The third assumption is a type of growth condition on the gradient and is satisfied for many types of convex functions, including linear and quadratic functions. Furthermore, an \(L_a\)-smooth convex function with a global lower bound $a_i(x) \geq a_i^\ast$ for all $x \in \bbR^n$ satisfies the assumption with \(C_1=2L_a\). %
We additionally use the notation $A(x) := (a_1(x),...,a_m(x))^T$ to refer to the vector of constraint functions throughout. 

The following lemma
shows that under Assumptions \ref{assum:Assumption Nonlinear Obj} and \ref{assum:Assumption Nonlinear Constraints}, given a large enough \(\xi\), any solution of the penalty reformulation \eqref{eq:Penalty Nonlinear} at \(\delta=0\), \(x_{\xi ,0}^*\), is a solution of the original problem \eqref{eq:Original Nonlinear}. In particular, the lemma shows that a sufficiently large value of $\xi$ is related to the $\infty$-norm of a dual multiplier vector $\lambda \in \bbR^m$. We use \(\Lambda^\ast\) to denote the set of optimal dual solutions of the Lagrangian dual of \eqref{eq:Original Nonlinear}, i.e., the set of dual multiplier vectors corresponding to the KKT conditions. %
\begin{lemma}\label{lemma:0 point nonlinear}
Define \(\bar{\xi} := \inf_{\lambda^\ast \in \Lambda^\ast}\|\lambda^\ast\|_\infty\). Then, $\bar{\xi}$ is finite and it holds that \(x_{\xi ,0}^*\) is an optimal solution to problem \eqref{eq:Original Nonlinear} for all \(\xi > \bar{\xi}\). Furthermore, in the strongly convex case ($\mu > 0$), it holds that \(x_{\xi ,0}^*\) is the unique optimal solution to problem \eqref{eq:Original Nonlinear} for all \(\xi \ge \bar{\xi}\).
\end{lemma}
\begin{proof}
First consider the case of \(\mu > 0\) and where \(\xi \geq \bar{\xi}\). By Slater condition (Assumption \ref{assum:Assumption Nonlinear Constraints}.2), the Lagrangian necessary conditions hold, i.e., there exists \(\lambda^*\in\bbR_{+}^{m}\) such that 
    \[0\in\partial F(x^*)+\sum_{i=1}^{m}\lambda_i^* \nabla a_i(x^*) ,\]
    and $a_i(x) = 0$ for all $i$ with $\lambda_i^* > 0$. Hence, $\Lambda^\ast$ is non-empty and $\bar{\xi}$ is finite.
    Given \(\bar{\xi} = \inf_{\lambda^\ast \in \Lambda^\ast}\|\lambda^\ast\|_\infty\), since $\partial_x p_0(a_i(x)) = \partial_t p_{0}(a_i(x))\nabla a_i(x)$ and \(\partial_t p_{0}(0)=[0,1]\), when \(\xi \ge \bar{\xi}\) it holds that
    \[0\in \partial F(x^*)+\xi \sum_{i=1}^{m}\partial_x p_0(a_i(x^*))=\partial_x F_{\xi ,0}(x^*).\]
    Thus, \(x^*\) is the unique solution of \eqref{eq:Penalty Nonlinear}.
The case of \(\mu=0\) and \(\xi > \bar{\xi}\) follows as a corollary of Proposition~\ref{prop:nonstrongly convex estimation} (the proof of which does not rely on this lemma).
\end{proof}

Notice that the lower bound on \(\xi \) in Lemma \ref{lemma:0 point nonlinear} is based on the minimal infinity norm of a dual optimal solution $\lambda^\ast$. \cite{mishchenko2018stochastic} present a bound on their penalty parameter based on an assumption that is close to the definition of Hoffman's constant \citep{guler1995approximations}. Although estimating both bounds may be difficult in some situations, we expect that our bound on the penalty parameter \(\bar \xi\) may be easier to estimate in practice since it relies directly on dual information and approximate dual solutions can be generated by an algorithm.  Furthermore, as a local property specific to a given problem instance, the norm of the dual solution may provide a better bound than Hoffman-type constants, which are defined based on the global properties of the feasible region across all instances.

\subsection{Bounds for the Penalty Reformulation}
In this subsection, our main goal is to estimate how well solutions of the penalty reformulations \eqref{eq:Penalty Nonlinear} approximate solutions of the original problem \eqref{eq:Original Nonlinear}. In the results below and throughout the paper, the \(O\) notation hides universal constants, and the \(\tilde O\) notation hides universal constants and logarithmic dependencies on the problem parameters. We define the following notion of $q$-norm approximate solutions of problem \eqref{eq:Original Nonlinear}, where $q$-norm is the standard $\ell_q$-norm on $\bbR^n$.

\begin{definition}\label{def:approx sol}
The point \(\tilde x\) is an \(\left(\epsilon_A,\epsilon_F\right)\)-approximate solution of \eqref{eq:Original Nonlinear} with respect to the $q$-norm if 
\begin{equation}\label{eq:approx sol}
\|\bar A(\tilde {x})\|_q \le \epsilon_A \text{ and } F(\tilde x) - F^\ast \le \epsilon_F,    
\end{equation}
where \(\bar A (\tilde{x}) =(\max(0,a_1(\tilde x)),...,\max(0,a_m(\tilde x)))^T\).
\end{definition}

For constrained problems like \eqref{eq:Original Nonlinear}, the norm of violations of constraints is an important measure to assess the approximation quality of a candidate solution $\tilde{x}$.
In our results, we consider the case of $q = 1$, as well as the case of $q = 2$, which leads to tighter bounds in the strongly convex case. 
First, Proposition \ref{prop:function value gap nonlinear} bounds the gap in objective value \(F(x_{\xi,\delta}^*)-F^*\). 

\begin{proposition}\label{prop:function value gap nonlinear}
For $\delta \geq 0$ and \(\xi > \bar \xi\), it holds that
\begin{equation*}
    - \xi\|\bar{A}(x_{\xi,\delta}^*)\|_1 \leq F(x_{\xi,\delta}^*) - F^\ast \leq m\xi\delta\log2.
\end{equation*}
Furthermore, in the strongly convex case ($\mu > 0$), the above inequality holds for \(\xi = \bar \xi\) as well.
\end{proposition}
\begin{proof}

The first inequality follows from Lemma \ref{lemma:0 point nonlinear}, in particular \(F_{\xi,0}(x^*)\le F_{\xi,0}(x_{\xi,\delta}^*)\). The second follows from \(F(x_{\xi,\delta}^*) \le F_{\xi,\delta}(x_{\xi,\delta}^*)\le F_{\xi,\delta}(x^*)\le F(x^*)+m\xi\delta\log 2\), where the last inequality is by the fifth property in 
\edits{Proposition 2.1 of \cite{li2022new}.} %
\end{proof}

Next, we bound the violation of constraints for the optimal solutions to the penalty reformulations. Here, Proposition \ref{prop:nonstrongly convex estimation} shows a simple result without strong convexity. 

\begin{proposition}\label{prop:nonstrongly convex estimation}
Consider the general convex case ($\mu = 0$) in the first part of Assumption \ref{assum:Assumption Nonlinear Obj}. Let \(X^*\) denote the set of optimal solutions to problem \eqref{eq:Original Nonlinear}, let \(E=\{i\in\{1,...,m\}: a_i(x^*)=0,\ \forall x^*\in X^*\}\) be the set of indices of active constraints for all \(x^*\in X^*\), 
let \(m^A\) be the cardinality of \(E\), and \(I = \{1,...,m\}/E\). Then, if \(\xi>\bar \xi\), for any \(\hat x\in\mathbb{R}^n\) it holds that
\[\begin{aligned}
(\xi - \bar{\xi})\|\bar{A}(\hat{x})\|_1 &\leq \xi \sum_{i\in I} \max(0,a_i(\hat x)) \\ 
& \ \ \ \ +(\xi-\bar \xi)\sum_{i\in E}\max(0,a_i(\hat x))\\
&\le F_{\xi,\delta}(\hat x)- F_{\xi,\delta}(x^*) + m\xi\delta\log2.
\end{aligned}\]
\end{proposition}

\begin{proof}
Define $\Delta_{\xi, \delta}^F := F_{\xi,\delta}(\hat x) - F_{\xi,\delta}(x^*)$ and let \(\lambda^*\in \arg\min_{\lambda\in \Lambda^*} \|\lambda\|_{\infty}\). Then, we have
\[\begin{aligned}
&\Delta_{\xi, \delta}^F = F(\hat x) - F(x^*) + \xi\sum_{i=1}^m (p_{\delta}(a_i(\hat x))-p_{\delta}(a_i(x^*)))\\
&\ge  -\sum_{i\in E}\lambda_i^* \nabla a_i(x^*)^T(\hat x-x^*) +\xi\sum_{i=1}^m (p_{\delta}(a_i(\hat x))-p_{\delta}(a_i(x^*)))\\
&\ge -\sum_{i\in E}\lambda_i^* (a_i(\hat x)-a_i(x^*))+\xi\sum_{i=1}^m (p_{\delta}(a_i(\hat x))-p_{\delta}(a_i(x^*)))\\
&\ge \xi \sum_{i\in I} \max(0,a_i(\hat x))+(\xi-\|\lambda^*\|_{\infty})\sum_{i\in E}\max(0,a_i(\hat x))\\
&\ \ \ \ -m\xi\delta\log2,
\end{aligned}\]
where the first inequality comes from the KKT condition \(0\in \partial F(x^*) + \sum_{i=1}^m\lambda_i^* \nabla a_i(x^*)\), the second comes from the convexity of \(a_i\), and the third comes from the first and fifth properties of \(p_\delta\) listed in Proposition 2.1 of \cite{li2022new}.
\end{proof}

\edits{The following example shows that the bounds on the function value gap and the 1-norm of constraint violations are tight for the general convex setting. 
\begin{example}\label{example:entrywise2}
Consider the following problem:
\[
\begin{aligned}
\min_{x\in\mathbb{R}^n}\quad&F(x)=\sum_{i=1}^m x_i\\
s.t.\quad & e_i^Tx\ge0,\ i=1,...,m,
\end{aligned}
\]
where \(m\le n\). Then, $\bar{\xi} = 1$, and for \(\xi\in(1,2)\) and \(\delta>0\), \(x_{\xi,\delta}^*\), the solution to the penalized problem 
\[\min_{x}\quad \sum_{i=1}^m x_i+\xi\sum_{i=1}^m\delta \log(1+\exp(-x_i/\delta))\]
satisfies \(x_{\xi,\delta,i}^* = \delta\log (\xi-1)\), \(i=1,...,m\). Hence, 
the 1-norm of constraint violations satisfies
$\|\bar{A}(x_{\xi,\delta}^*)\|_1 = m\delta\log (\xi-1) = \Theta( m \delta)$
and the objective function gap satisfies \(F(x_{\xi,\delta}^*) - F(x^\ast) = m\delta\log (\xi-1) = \Theta(m\delta)\).
\end{example}}

Under the assumption of strong convexity of $F$, we establish our main result in Theorem \ref{thm:estimation nonlinear new z infty m 01}. Theorem \ref{thm:estimation nonlinear new z infty m 01} provides a stronger and more applicable bound on the 2-norm of the constraint violations at \(x_{\xi,\delta}^*\). Combined with Proposition \ref{prop:function value gap nonlinear}, Theorem \ref{thm:estimation nonlinear new z infty m 01} shows that \(x_{\xi,\delta}^*\) is a \(\tilde O(m\delta)\) solution to the original problem simultaneously in terms of both objective function value gap and violations of the constraints.
\begin{theorem}\label{thm:estimation nonlinear new z infty m 01}
Suppose \(\mu>0\). Let \({U} = 2 C_1\|A(x^*)\|_{\infty}+2m C_0\) and let \(\xi \ge \bar \xi\) be given. Then, for all \(\delta\in [0,\frac{ {U}\xi }{\mu}\exp(-2)]\) satisfying $2mC_1\delta\log\left(\tfrac{{U}\xi}{\mu\delta}\right)\le {U}$, it holds that
\[\|\bar A(x_{\xi ,\delta}^*)\|_2 \leq \|A(x_{\xi ,\delta}^*)-A(x^*)\|_2\le \sqrt{m}\delta\log\left(\tfrac{{U}\xi }{\mu\delta}\right).\]
\end{theorem}
Before presenting the proof of Theorem \ref{thm:estimation nonlinear new z infty m 01}, we remark on the comparison with the more basic result of Proposition \ref{prop:nonstrongly convex estimation}.
\begin{remark}
Notice that, in the general non-strongly convex case, a direct corollary of Proposition \ref{prop:nonstrongly convex estimation} is that, when \(\xi\ge 2\inf_{\lambda^*\in \Lambda^*}\|\lambda^* \|_{\infty}\ge 0\) and \(F_{\xi,\delta}(\hat x)-F_{\xi,\delta}(x_{\xi,\delta}^*)\le m\xi\delta\log2\), then it holds that
\[\|\bar A(\hat x)\|_1 \le 4m\delta\log2.\]
The above bound is a weaker result compared with Theorem~\ref{thm:estimation nonlinear new z infty m 01}, since an \(\tilde O(\sqrt{m}\delta)\)  2-norm bound implies an \(\tilde O(m\delta)\) 1-norm bound.
\end{remark}

\begin{proof}

{\em Part 1:} First, we prove the results for the special case where \(F\) and \(a_i\)'s are all twice-continuously differentiable. In this part, for simplicity of notation, we consider \(\xi > \bar \xi\) be a constant and use the simpler notations \(x_{\delta}^*\), \(z_{\delta}\), \(\tilde z_{\delta}\) in the proof. Consider \(x_{\delta}^*\) as a function of \(\delta\). Since \(\nabla F_{\xi ,\delta}(x_{\delta}^*)=0\), by taking the derivative w.r.t. \(\delta\) and applying the chain rule and total derivative formula we have
\[\nabla^2 F_{\xi ,\delta}(x_{\delta}^*)\frac{d x_{\delta}^*}{d\delta}+\frac{\partial}{\partial\delta}\nabla F_{\xi ,\delta}(x_{\delta}^*)=0.\]
By the strong convexity of $F$, we thus have
\begin{equation}\label{eq:Multiple Constraints Derivative Nonlinear}
\begin{aligned}
\frac{d x_{\delta}^*}{d\delta}=\left(\nabla^2 F_{\xi ,\delta}(x)\right)^{-1}\frac{\partial}{\partial\delta}\nabla F_{\xi ,\delta}(x).
\end{aligned}
\end{equation}
Now, note that
\[\begin{aligned}
\nabla F_{\xi ,\delta}(x)&=\nabla F(x)+\xi \sum_{i=1}^m\nabla p_{\delta}(a_i(x))\\ &= \nabla F(x)+\xi \sum_{i=1}^m \nabla a_i(x) \tfrac{\exp(a_i(x)/\delta)}{1+\exp(a_i(x)/\delta)} \cdot
\end{aligned}\]
Then,
\[\frac{\partial}{\partial\delta}\nabla F_{\xi ,\delta}(x)=\xi \sum_{i=1}^m \nabla a_i(x) \tfrac{-[a_i(x)/\delta^2]\exp(a_i(x)/\delta)}{(1+\exp(a_i(x)/\delta))^2}, \ \text{and} \]
\begin{equation}\label{eq:Hessian, linear, infty, m}
\begin{aligned}
    \nabla^2 F_{\xi ,\delta}(x) &= \nabla^2 F(x)+\xi \sum_{i=1}^m \tfrac{[\nabla a_i(x)\nabla a_i(x)^T/\delta]\exp(a_i(x)/\delta)}{(1+\exp(a_i(x)/\delta))^2}\\ 
    & \ \ \ \ + \xi \sum_{i=1}^m \nabla^2 a_i(x) \tfrac{ \exp(a_i(x)/\delta)}{1+\exp(a_i(x)/\delta)}.    
\end{aligned}
\end{equation}
We define the following new notation: \(\tilde{z} =  A(x^*)\), \(\tilde z_{\delta}=A(x_{\delta}^*)\), \(z_{\delta} = \tilde z_{\delta}-\tilde z_{0}=\tilde z_{\delta}-\tilde z\), and
\begin{align}\label{eqn_proof_notation}
&\Pi_{\tilde z_{\delta},\delta}=\mathrm{diag} \left(\tfrac{\exp( a_1(x_{\delta}^*)/\delta)}{(1+\exp( a_1(x_{\delta}^*)/\delta))^2},\cdots,\tfrac{\exp( a_m(x_{\delta}^*)/\delta)}{(1+\exp( a_m(x_{\delta}^*)/\delta))^2} \right), \nonumber \\
&A_{\delta} = (\nabla a_1(x_{\delta}^*),\cdots,\nabla a_m(x_{\delta}^*))^T,\\
&H_{\delta} = \nabla^2F(x_{\delta}^*)+\xi\sum_{i=1}^m \nabla^2 a_i(x_{\delta}^*) \tfrac{ \exp(a_i(x_{\delta}^*)/\delta)}{1+\exp(a_i(x_{\delta}^*)/\delta)}. \nonumber
\end{align}
Then, combining the above notation with \eqref{eq:Hessian, linear, infty, m} allows us to express \eqref{eq:Multiple Constraints Derivative Nonlinear} as
\begin{equation}\label{eq:Multiple Constraints Derivative Nonlinear Matrix Version}
\frac{dx_{\delta}^*}{d \delta}=\xi (\delta H_{\delta} + \xi A_{\delta}^T\Pi_{\tilde z_{\delta},\delta}A_{\delta})^{-1}A_{\delta}^T\Pi_{\tilde z_{\delta},\delta}\frac{\tilde z_{\delta}}{\delta} \cdot    
\end{equation}
Then, 
\[
\frac{dz_{\delta}}{d \delta}=\xi A_{\delta}(\delta H_{\delta} + \xi A_{\delta}^T\Pi_{\tilde z_{\delta},\delta}A_{\delta})^{-1}A_{\delta}^T\Pi_{\tilde z_{\delta},\delta}\frac{\tilde z_{\delta}}{\delta} \cdot  
\]
By Woodbury's identity,
\[\begin{aligned}
&(\delta H_{\delta} + \xi A_{\delta}^T\Pi_{\tilde z_{\delta},\delta}A_{\delta})^{-1} = \delta^{-1}H_{\delta}^{-1} ~-~ \\ & \ \ \ \ \delta^{-2}H_{\delta}^{-1}A_{\delta}^T(\delta^{-1}A_{\delta}H_{\delta}^{-1}A_{\delta}^T+\xi ^{-1}\Pi_{\tilde z_{\delta},\delta}^{-1})^{-1}A_{\delta} H_{\delta}^{-1};
\end{aligned}\]
so let \(J_{\delta}=A_{\delta}\left(\tfrac{H_{\delta}}{\mu}\right)^{-1}A_{\delta}^T\), then
\[\begin{aligned}
\frac{dz_{\delta}}{d \delta}
=J_{\delta} \left(J_{\delta}+\frac{\mu\delta}{\xi }\Pi_{\tilde z_{\delta},\delta}^{-1}\right)^{-1} \frac{\tilde z_{\delta}}{\delta}.
\end{aligned}\]

Notice that 
\[\begin{aligned}
J_{\delta} &= A_{\delta}\left(\tfrac{H_{\delta}}{\mu}\right)^{-1}A_{\delta}^T \preceq \sum_{i=1}^m\|\nabla a_i(x_{\delta}^*)\|_2^2 I\\
&\preceq (m C_0 + C_1\|\tilde z_{\delta}\|_{\infty})I\\ &\preceq (m C_0 + C_1\|\tilde z_{0}\|_{\infty} + \sqrt{m} C_1\|z_{\delta}\|_2)I.
\end{aligned}\]
Define 
\begin{equation*}
\Delta_{\max}=\sup \bigg \{\delta\in[0,\frac{ {U}\xi }{\mu}\exp(-2)]:\ 2mC_1\delta\log\left(\tfrac{{U}\xi}{\mu\delta}\right)\le {U} \bigg \}.    
\end{equation*}
We now show that $m C_0 + C_1\|\tilde z_{0}\|_{\infty} + \sqrt{m} C_1\|z_{\delta}\|_2 \leq U$ for all $\delta \in [0, \Delta_{\max}]$. Let \(\bar \delta = \sup \{\delta\le \frac{{U}\xi }{\mu}\exp(-2): \ 2\sqrt{m} C_1\|z_{\delta}\|_2\le {U}\}\). Then, for \(\delta \in [0,\bar \delta]\),  we have
\[\left\|\frac{dz_{\delta}}{d \delta}\right\|_2\le\left\|\left(\tfrac{|\tilde z_{\delta,1}|/\delta}{1+\tfrac{\mu\delta}{{U}\xi }\exp(|\tilde z_{\delta,1}|/\delta)},...,\tfrac{|\tilde z_{\delta,m}|/\delta}{1+\tfrac{\mu\delta}{{U}\xi }\exp(|\tilde z_{\delta,m}|/\delta)}\right)\right\|_2.\]
And since \(\bar \delta \le \frac{{U}\xi }{\mu}\exp(-2)\), it holds that
\[\tfrac{|\tilde z_{\delta,i}|/\delta}{1+\tfrac{\mu\delta}{{U}\xi }\exp\left(|\tilde z_{\delta,i}|/\delta\right)}\le \log\left(\tfrac{{U}\xi }{\mu\delta}\right)-1.\]
Hence, \(\left\|\tfrac{dz_{\delta}}{d \delta}\right\|_2\le \sqrt{m}\left(\log\left(\tfrac{{U}\xi }{\mu\delta}\right)-1\right)\), and  \(\|z_{ \delta}\|_2\le\sqrt{m} \delta\log(\tfrac{{U}\xi }{\mu \delta})\), i.e.,
\begin{equation*}
2\sqrt{m}C_1\|z_{\delta}\|\le 2mC_1\delta\log\left(\tfrac{U\xi}{\mu\delta}\right)\) for \(\delta\in [0,\bar \delta].
\end{equation*}
Therefore, \(\bar \delta\le\Delta_{\max}\), and \(\|z_{\delta}\|_2\le \sqrt{m} \delta\log(\tfrac{{U}\xi }{\mu \delta})\) for \(\delta\in [0,\Delta_{\max}]\).

{\em Part 2:} Now suppose some (or all) \(F\) and \(a_i\)'s are not twice-continuously differentiable. For \(\epsilon>0\), let \(S_{\epsilon}\) be the following sublevel set
\[S_{\epsilon}=\{x\in\bbR^n:F_{\xi,0}(x)\le F_{\xi,0}(0)+m\xi\Delta_{\max}+(2m\xi+1)\epsilon\}.\]
Since 
\[\begin{aligned}
&F_{\xi,0}(x_{\xi,\delta}^*)=F(x_{\xi,\delta}^*)+\xi\sum_{i=1}^{{m}}\max(0,a_i(x_{\xi,\delta}^*))\le F_{\xi,\delta}(x_{\xi,\delta}^*)\\\le &F_{\xi,\delta}(0)
\le F(0)+\xi\sum_{i=1}^{{m}}\max(0,a_i(0))+m\xi\Delta_{\max},
\end{aligned}\]
we have \(x_{\xi,\delta}^*\in S_{\epsilon}\) for all \(\delta\in[0,\Delta_{\max}]\). Because \(F\) is strongly convex, \(S_{\epsilon}\) is a compact convex set. Then, there exists \(R_{\epsilon}>0\), s.t. \(\forall x\in S_{\epsilon}\), \(\|x\|_2\le R_{\epsilon}\). Hence, we construct the following approximation of \(a_i(x)\), \(i=1,...,{m}\),
\[\tilde a_i(x)=a_i(x)+\frac {\epsilon}{2}\|x\|_2^2\min\left(1,\frac{1}{R_{\epsilon^2}}\right).\]
Then, by Theorem 1.9 and 1.10 in \cite{azagra2013global}, there are \(C^{\infty}\) convex approximations \(\hat F(x)\) and \(\hat a_i(x)\), \(i=1,...,{m}\), such that $\hat F(x)\text{ is }\mu-\text{stongly convex}$ and
\[\begin{aligned}
&F(x)-\epsilon\le \hat F(x)\le F(x),  \\ 
&\|\nabla \hat a_i(x)-\nabla \tilde a_i(x) \|_2\le\epsilon, \text{ and } \\
&|\hat a_i(x)-\tilde a_i(x) |\le\epsilon,\ \forall x\in\bbR^n.
\end{aligned}\]
Let \[\hat F_{\xi,\delta}(x)=\hat F(x)+\xi\sum_{i=1}^{{m}}p_{\delta}(\hat a_i(x)),\]
and \(\hat x_{\xi,\delta}^*=\arg\min\hat F_{\xi,\delta}(x)\). The above approximation guarantees imply that 
\[\begin{aligned}
&\|\nabla \hat a_i(x)\|_2^2 \le (\epsilon+\|\nabla a_i(x)\|_2)^2 \\
& \le \epsilon^2+2\epsilon\sqrt{C_0+C_1|a_i(x)|}+C_0+C_1|a_i(x)|\\
& \le \epsilon^2 + \epsilon + (C_0+C_1|a_i(x)|) (1+\epsilon)\\
& \le \hat C_0 + \hat C_1 |\hat a_i(x)|,
\end{aligned}\]
where \(\hat C_1 = C_1(1+\epsilon)\), and \(\hat C_0 = (1+\epsilon) ( C_0 + \epsilon C_1 + \epsilon )\). Let \(\tilde z_{\xi,\delta}'=(\hat a_1(x_{\xi,\delta}^*),...,\hat a_m(\hat x_{\xi,\delta}^*))\), \(\hat z_{\xi,\delta} = \tilde z_{\xi,\delta}'-\tilde z_{\xi,0}\). By the result of the first part, we know for \(\delta \in [0,\hat \Delta_{\max}]\), 
\begin{equation}\label{eq:hat est}
\|\hat z_{\xi,\delta}\|_2\le \sqrt{m}\delta \log(\tfrac{\hat U \xi}{\mu\delta}),
\end{equation} where 
\(\hat U = 2 \hat C_1 \|\tilde z'_{\xi,0}\|_{\infty} +2m\hat C_0\), and \(\hat \Delta_{\max}=\sup\ \big \{\delta\in[0,\frac{ \hat {U}\xi }{\mu}\exp(-2)]:\ 2mC_1\delta\log\left(\tfrac{\hat {U}\xi}{\mu\delta}\right)\le \hat {U} \big \}\ge \Delta_{\max}\).

\noindent
Then, 
\[\begin{aligned}
&F(\hat x_{\xi,\delta}^*)+\xi\sum_{i=1}^{{m}}\max(0,a_i(\hat x_{\xi,\delta}^*))\\ \le\ &F(\hat x_{\xi,\delta}^*)+\xi\sum_{i=1}^{{m}}\max(0,\hat a_i(\hat x_{\xi,\delta}^*))\\ 
\le\ & \epsilon + \hat F(\hat x_{\xi,\delta}^*)+m\xi\epsilon+\xi\sum_{i=1}^{{m}}\max(0,\hat a_i(\hat x_{\xi,\delta}^*))\\
\le\ & (m\xi+1)\epsilon+ \hat F(0)+\xi\sum_{i=1}^{{m}}\max(0,\hat a_i(0))+m\xi\delta\\
\le\ & (2m\xi+1)\epsilon+ F(0)+ \xi\sum_{i=1}^{{m}}\max(0,a_i(0))+m\xi\delta.\\
\end{aligned}\]
Hence, when \(\delta\in[0,\Delta_{\max}]\), we have \(\hat x_{\xi,\delta}^*(\epsilon)\in S_{\epsilon}\). Then, because 
\[\begin{aligned}
&|F_{\xi,\delta}(x)-\hat F_{\xi,\delta}(x)|\\\le\ &|F(x)-\hat F(x)|+\xi\sum_{i=1}^{{m}}|p_{\delta}(a_i(x))-p_{\delta}(\hat a_i(x))|\\
\le\ & (2m\xi+1)\epsilon,\ \forall x \in S_{\epsilon},
\end{aligned}\]
by the \(\mu\)-strong convexity, we have \(\|x_{\xi,\delta}^*-\hat x_{\xi,\delta}^*\|_2 \le  \sqrt{\frac{2(2m\xi+1)\epsilon}{\mu}}\). Therefore, 
\[\begin{aligned}
|a_i(x_{\xi,\delta}^*)-\hat a_i(\hat x_{\xi,\delta}^*)| \le \epsilon + |a_i(x_{\xi,\delta}^*) - a_i(\hat x_{\xi,\delta}^*)|
\end{aligned}\]
 converges to \(0\) when \(\epsilon\) goes to 0. Combined with \eqref{eq:hat est}, since \(\epsilon\in(0,\epsilon_{\max}]\) can be chosen arbitrarily, when \(\epsilon\) goes to $0$, we have the following upper bound:
 \[ \|A(x_{\xi ,\delta}^*)-A(x^*)\|_2\le \sqrt{m}\delta\log\left(\tfrac{{U}\xi }{\mu\delta}\right).\]
\end{proof}

\begin{remark}
Following the same proof technique as in Theorem~\ref{thm:estimation nonlinear new z infty m 01}, applied to the change in function values, leads to an \(\tilde O(m\xi\delta)\) gap. Furthermore, unlike Proposition \ref{prop:function value gap nonlinear}, this bound can be tightened to \(\tilde O(m^A\xi\delta)\) where $m^A$ is the number of active constraints at $x^*$.
Indeed, notice that 
\[\nabla F_{\xi ,\delta}(x_{\xi,\delta}^*) = \nabla F(x_{\xi,\delta}^*)+\xi \sum_{i=1}^m \nabla a_i(x_{\xi,\delta}^*) \tfrac{\exp(a_i(x_{\xi,\delta}^*)/\delta)}{1+\exp(a_i(x_{\xi,\delta}^*)/\delta)}=0,\]
which, recalling the notation defined in \eqref{eqn_proof_notation}, can be rearranged to
\[\begin{aligned}
\nabla F(x_{\xi,\delta}^*)^T = &\xi e^T \mathrm{diag}  \left(p_{\delta}'(a_1(x_{\xi,\delta}^*)),...,p_{\delta}'(a_m(x_{\xi,\delta}^*)) \right)A_{\delta}    
\end{aligned}.\]
Hence,
\[\begin{aligned}
& \frac{d F (x_{\xi,\delta}^*) }{d \delta} = \nabla F(x_{\xi,\delta}^*)^T \frac{d x_{\xi,\delta}^*}{d \delta}\\
& = \xi e^T \mathrm{diag}  \left(p_{\delta}'(a_1(x_{\xi,\delta}^*)),...,p_{\delta}'(a_m(x_{\xi,\delta}^*)) \right)\\&A_{\delta}
(\delta H_{\delta} + \xi A_{\delta}^T\Pi_{\tilde z_{\delta},\delta}A_{\delta})^{-1}A_{\delta}^T\Pi_{\tilde z_{\delta},\delta}\frac{\tilde z_{\delta}}{\delta} .
\end{aligned}\]

Following similar steps as in the proof of Theorem~\ref{thm:estimation nonlinear new z infty m 01} leads to the bound:
\[\left|\frac{d F (x_{\xi,\delta}^*) }{d \delta}\right| \le  \xi\sum_{i=1}^m \tfrac{|\tilde z_{\delta,i}|/\delta}{1+\tfrac{\mu\delta}{{U}\xi }\exp(|\tilde z_{\delta,i}|/\delta)}\tfrac{\exp(\tilde z_{\delta,i}/\delta)}{1+\exp(\tilde z_{\delta,i}/\delta)}.\]
And similar as Proposition~\ref{prop:function value gap nonlinear}, the upper bound above is \(\tilde O(m\xi\delta)\) as \(\delta\) goes to 0. 

Notice that for inactive constraints \(a_i\), (i.e. \(a_i(x^*)<0\)), we have that \(\tfrac{\exp(\tilde z_{\delta,i}/\delta)}{1+\exp(\tilde z_{\delta,i}/\delta)}\) goes to zero when \(\delta\) goes to zero (since \(\tilde z_{\delta,i}\) will not go to zero). Therefore, asymptotically the gap converges in an \(\tilde O(m^A \xi\delta)\) rate, where $m^A$ is the number of active constraints at $x^*$. %
\end{remark}

The following example shows that the bound in Theorem~\ref{thm:estimation nonlinear new z infty m 01} is tight. 

\begin{example}\label{example:entrywise1}
Consider the following problem:
\[
\begin{aligned}
\min_{x\in\mathbb{R}^n}\quad&F(x)=\frac12\|x\|_2^2+e^Tx\\
s.t.\quad & e_i^Tx\ge0,\ i=1,...,m,
\end{aligned}
\]
where \(m\le n\). Then, $\bar{\xi} = 1$, and for \(\xi\in[1,2)\) and \(\delta>0\), \(x_{\xi,\delta}^*\), the solution to the penalized problem 
\[\min_{x}\quad \frac12\|x\|_2^2+e^Tx+\xi\sum_{i=1}^m\delta \log(1+\exp(-x_i/\delta))\]
satisfies for each component  \(i=1,...,m\), that \(x_{\xi,\delta,i}^*\) is given by the solution to \(1+t-\xi\tfrac{\exp(-\frac t\delta)}{1+\exp(-\frac t\delta)}=0\), which is \(\tilde \Theta(\delta)\), \(i=1,...,m\). Hence, 
the 2-norm of constraint violations satisfies
$\|\bar{A}(x_{\xi,\delta}^*)\|_2 = \tilde \Theta( \sqrt{m} \delta)$
and the objective function gap satisfies \(F(x_{\xi,\delta}^*) - F(x^\ast)= \tilde \Theta(m\delta)\).
\end{example}

\section{Algorithms and Applications}
In this section, we discuss several applications of the penalty reformulation \eqref{eq:Penalty Nonlinear} and our estimation results (primarily Theorem \ref{thm:estimation nonlinear new z infty m 01}) for the error bounds in the previous section. First, we consider solving a static penalty reformulation with an accelerated stochastic gradient method and analyze the convergence rate towards an approximate solution (Definition \ref{def:approx sol}) of the original constrained problem. Next, we show how we can address non-smooth constraints with smooth approximations.
\label{sec:applications}
\subsection{Algorithms and Complexity}
\label{subsec:complexity}
In this subsection, we show how we can solve the original problem \eqref{eq:Original Nonlinear} by applying an accelerated stochastic gradient method to our penalty reformulation \eqref{eq:Penalty Nonlinear}. For simplicity, in Corollary~\ref{cor:Static Complexity-SVRG Catalyst Nonlinear}, we focus on using constant \(\xi\ge \bar \xi\) and \(\delta\) tuned based on the desired accuracy, and we solve the corresponding penalty reformulation problem \(\min_{x} F_{\xi,\delta}(x)\) with accelerated deterministic or stochastic gradient methods.  We state Corollary~\ref{cor:Static Complexity-SVRG Catalyst Nonlinear} with the SVRG with catalyst acceleration \citep{lin2015universal, lin2018catalyst}, while other accelerated methods for finite sum problems, such as Katyusha \citep{allen2017katyusha} and the RPDG method \citep{lan2018optimal}, are also applicable for this result.
Before showing the complexity results, we show that the smoothness of \eqref{eq:Penalty Nonlinear} can be estimated, which means using first-order methods on \eqref{eq:Penalty Nonlinear} is practical. %

\begin{proposition}\label{prop:LSmoothforPenaltyReformulation}
\(p_{\delta}(a_i(x))\) is \((L_a+\frac{C_1}{4}+\frac{C_0}{4\delta})\)-smooth, \(i=1,...,{m}\).
\end{proposition}

\begin{proof}
For \(x_1,x_2\in \bbR^n\), without loss of generality, assume \(a_i(x_2)\ge a_i(x_1)\). Then, there exists a point \(x_c\) on the line segment between $x_1$ and $x_2$, s.t. 
\[\begin{aligned}
\nabla a_i(x_c)^T(x_2-x_1) p''_{\delta}(a_i(x_c))= p'_{\delta}(a_i(x_2))-p'_{\delta}(a_i(x_1)),
\end{aligned}\]
where 
\[\begin{aligned}
p_{\delta}'(t) = \frac{\exp(t/\delta)}{1+\exp(t/\delta)},\ p_{\delta}''(t) = \frac{\exp(t/\delta)}{\delta (1+\exp(t/\delta))^2} \cdot
\end{aligned}\]
Then,
\[\begin{aligned}
&\|\nabla p_{\delta}(a_i(x_2))-\nabla p_{\delta}(a_i(x_1))\|_2\\
& = \|p'_{\delta}(a_i(x_2))\nabla a_i(x_2)-p'_{\delta}( a_i(x_1) )\nabla a_i(x_1) \|_2\\
& \le \|\nabla a_i(x_c)\|_2 (p'_{\delta}(a_i(x_2))-p'_{\delta}(a_i(x_1)))\\
& \ \ \ \ + p'_{\delta}(a_i(x_2))\|\nabla a_i(x_2)-\nabla a_i(x_c)\|_2\\
& \ \ \ \ + p'_{\delta}(a_i(x_1))\|\nabla a_i(x_1)-\nabla a_i(x_c)\|_2\\
& \le \|\nabla a_i(x_c)\|_2^2p_{\delta}''(a_i(x_c))\|x_2-x_1\|_2+L_a\|x_2-x_1\|_2\\
& \le (L_a+(C_0+C_1a_i(x_c))p_{\delta}''(a_i(x_c)))\|x_2-x_1\|_2.
\end{aligned}\]
Notice that 
\[\begin{aligned}
\frac{\exp(t)}{(1+\exp(t))^2}&\le\frac14,\ t\frac{\exp(t)}{(1+\exp(t))^2}&\le \frac14 \cdot\\
\end{aligned}\]
Thus, we have the  result.
\end{proof}

\begin{corollary}\label{cor:Static Complexity convex}
Suppose \(\mu=0\). For any desired accuracy \(\epsilon\) and penalty parameter \(\xi \ge 2\bar{\xi}\), consider applying accelerated full gradient methods to solve a penalty reformulation subproblem \eqref{eq:Penalty Nonlinear} with penalty parameter
$\delta_{\epsilon}=\epsilon/4m\log 2.$
Then, the number of full-gradient iterations required to get an \(\left(\epsilon,\xi\epsilon\right)\)-approximate solution \(\tilde x\)  with respect to the $1$-norm
is upper bounded by
\[O\left(\frac{L_f+m(L_a+C_1)}{\sqrt{\epsilon}} + \frac{m^2\xi C_0}{\epsilon\sqrt{\epsilon}}\right) \cdot\]
\end{corollary}
\begin{proof}
Notice that when \(\xi\ge 2\bar{\xi}\), \(\delta_{\epsilon}=\frac{\epsilon}{4m\log 2}\), and \(F_{\xi,\delta_{\epsilon}}(\tilde x)-F_{\xi,\delta_{\epsilon}}(x_{\xi,\delta_{\epsilon}^*})\le m\xi\delta_{\epsilon}\log2\), by Proposition~\ref{prop:nonstrongly convex estimation},
\(\|\bar A(\tilde x)
\|_1\le \epsilon\). Similar to Proposition~\ref{prop:function value gap nonlinear},
\[\begin{aligned}&F(\tilde x) \le F_{\xi,\delta_{\epsilon}}(\tilde x) \\ &\le F_{\xi,\delta_{\epsilon}}(x_{\xi,\delta_{\epsilon}^*} )+ m\xi\delta_{\epsilon}\log2 \le F(x^*) +2m\xi\delta_{\epsilon}\log2\\&\le F(x^*)+\frac{\xi\epsilon}2 \cdot
\end{aligned}\]
When \(\delta_{\epsilon}=\frac{\epsilon}{4m\log 2}\), the problem consists of \({\ell}\) components that are \(\frac{{\ell}+{m}}{\ell}L_f\)-smooth, and \({m}\) components that are \((m+\ell)(L_a+\frac{C_1}{4}+\frac{C_0}{4\delta_{\epsilon}})\), i.e., \(O\left((m+\ell)(L_a+C_1+m\xi C_0/\epsilon)\right)\)-smooth. Hence, if we apply the accelerated method to solve \(\min_{x} F_{\xi,\delta_{\epsilon}}(x)\) such that \(F_{\xi,\delta_{\epsilon}}(\tilde x)-F_{\xi,\delta_{\epsilon}}^* \le m\xi\delta_{\epsilon}\log2=\frac{\xi\epsilon}4\), the complexity is
\[O\left(\frac{L_f+m(L_a+C_1)}{\sqrt{\epsilon}} + \frac{m^2\xi C_0}{\epsilon\sqrt{\epsilon}}\right).\]
\end{proof}

\begin{corollary}\label{cor:Static Complexity-SVRG Catalyst Nonlinear}
Suppose \(\mu>0\) and let \(\epsilon\in (0,\frac{\sqrt{m}{U}\xi }{\mu}\exp(-2)]\), with penalty parameters \(\xi \ge \bar{\xi}\), and \[\delta_{\epsilon}=\frac{\epsilon}{4{\sqrt m}}\left(\log\left(\frac{2{\sqrt m}{U}\xi }{\mu\epsilon}\right)\right)^{-1}.\]
1. Consider applying accelerated full gradient methods to solve a penalty reformulation subproblem \eqref{eq:Penalty Nonlinear} with the penalty parameters \(\xi\) and \(\delta_{\epsilon}\). . Then, the number of full-gradient iterations required to get an \(\left(\epsilon,\sqrt{m}\xi\epsilon\right)\)-approximate solution \(\tilde x\)  with respect to the $2$-norm is upper bounded by
\[
\tilde O\left(\sqrt{\frac{L_f+m\xi L_a + m\xi C_1}{\mu}}+m^{\frac34}\sqrt{\frac{\xi C_0}{\mu\epsilon}}\right) \cdot
\]
2. Consider applying proximal SVRG with catalyst acceleration \citep{lin2015universal, lin2018catalyst} to solve a penalty reformulation subproblem \eqref{eq:Penalty Nonlinear} with the penalty parameters \(\xi\) and \(\delta_{\epsilon}\). 
Then, the expected number of incremental steps required to obtain an \(\left(\epsilon,\sqrt{m}\xi\epsilon\right)\)-approximate solution \(\tilde x\) with respect to the $2$-norm is upper bounded by
\[
\tilde O\left(\ell + \sqrt{\frac{(\ell+m)(L_f+m\xi L_a + m\xi C_1)}{\mu}}+m^{\frac34}\sqrt{\frac{(\ell+m)\xi C_0}{\mu\epsilon}}\right).
\]
\end{corollary}

\begin{proof}
We only prove the 2nd claim as the proof for the 1st claim is similar. Since $\frac{1}{w}>2\log(\frac1w)$ for all \(w>0\), we have
\[\frac{w}{2\log(\frac1w)}\log\left(\tfrac{2\log(\frac1w)}w\right)\le w,\ \forall w\in(0,1).\]
Hence, for \(\delta_{\epsilon}=\frac{\epsilon}{4{\sqrt m}}\left(\log\left(\frac{2{\sqrt m}{U}\xi }{\mu\epsilon}\right)\right)^{-1}\), we have
\[2{\sqrt m}\delta_\epsilon\log\left(\frac{{U}\xi }{\mu\delta_{\epsilon}}\right)\le \epsilon.\]

To minimize \begin{equation}\label{eq:epsilon unconstrained problem}
    F_{\xi,\delta_{\epsilon}}(x)=\frac1{\ell}\sum_{i=1}^{\ell}f_i(x)+\xi\sum_{j=1}^m p_{\delta_{\epsilon}}(a_j(x))+\psi(x),
\end{equation} we consider \(\frac{{\ell}+{m}}{\ell} f_i(x)\), \(i=1,...,{\ell}\), and \(({\ell}+{m})\xi  p_{\delta_{\epsilon}}(a_j(x))\), \(j=1,...,{m}\) to be the ${\ell}+{m}$ components. Then, the problems are \(\mu\)-strongly convex, and consist of \({\ell}\) components that are \(\frac{{\ell}+{m}}{\ell}L_f\)-smooth, and \({m}\) components that are \((m+\ell)\xi(L_a+\frac{C_1}{4}+\frac{C_0}{4\delta})\)-smooth. Hence, if we apply the catalyst accelerated method to solve \eqref{eq:epsilon unconstrained problem} such that \(F_{\xi,\delta_{\epsilon}}(\tilde x)-F_{\xi,\delta_{\epsilon}}^* \le \Delta_{\epsilon}\), by Lemma C.1 of \cite{lin2015universal},
we can upper bound the expected number of incremental steps by 
\[
\tilde O\left(\ell + \sqrt{\frac{(\ell+m)(L_f+m\xi L_a + m\xi C_1)}{\mu}}+m^{\frac34}\sqrt{\frac{(\ell+m)\xi C_0}{\mu\epsilon}}\right).
\]
Here, \(\Delta_{\epsilon}\) is chosen to satisfy 
\[\begin{aligned}
&|a_i(\tilde x)-a_i(x_{\xi,\delta_{\epsilon}}^*)|\le \|\tilde x-x_{\xi,\delta_{\epsilon}}^*\|_2\|\nabla a_i(x_{\xi,\delta_{\epsilon}}^*)\|_2\\
&\le \sqrt{C_0+C_1|a_i(x_{\xi,\delta_{\epsilon}}^*)|}\|\tilde x-x_{\xi,\delta_{\epsilon}}^*\|_2\\
&\le \sqrt{C_0+C_1\sqrt{m}\delta\log\left(\frac{{U}\xi}{\mu\delta}\right)}\sqrt{\frac{\Delta_\epsilon}{\mu}}\\
&\le \frac \epsilon {2\sqrt m},
\end{aligned}\]
such that \(\|\bar A(\tilde x)\|_2\le \epsilon\)  by Theorem \ref{thm:estimation nonlinear new z infty m 01}.
The estimation for the function value gap is similar.

\end{proof}
Note that in the claims of Corollary~\ref{cor:Static Complexity-SVRG Catalyst Nonlinear}, the corresponding approximate solution has \(\epsilon_A=\epsilon\) and \(\epsilon_F=\sqrt{m}\xi\epsilon\) due to Theorem~\ref{thm:estimation nonlinear new z infty m 01} and Proposition~\ref{prop:function value gap nonlinear}. The factor of \(\xi\) comes from the equivalence of problem with scaling of constraints and dual variables, as \(\xi\) is related to the dual solutions.

The stochastic method based on the catalyst acceleration generally requires at least
\(\tilde O\left( \sqrt{{\ell}+{m}}\right)\) fewer evaluations on the components of the objective and the constraints and their gradients, 
compared to the accelerated full-gradient method, which indicates the advantage of applying stochastic methods over deterministic ones.

\begin{remark}\label{remark:linear comparison}
For linear constraints, i.e., \(L_a=C_1=0\), and w.l.o.g. \(C_0=1\), the corresponding complexity to achieve an \(\epsilon\) 2-norm violation of constraints in Corollary~\ref{cor:Static Complexity-SVRG Catalyst Nonlinear} is \(\tilde O\left(\ell + \sqrt{\frac{(\ell+m)L_f}{\mu}}+m^{\frac34}\sqrt{\frac{(\ell+m)\xi}{\mu\epsilon}}\right)\). Compared with the complexity to achieve an \(\epsilon\) error of solution in Corollary 3.7 of \cite{li2022new},  the third term is \(m^{\frac14}\) smaller, which is consistent with the difference of results in Theorem 2.5 and Proposition 2.8 in \cite{li2022new}. Despite this consistence, Corollary \ref{cor:Static Complexity-SVRG Catalyst Nonlinear} generalizes the complexity results to convex smoothly constrained problems as a simpler version.
\end{remark}
Corollary~\ref{cor:Static Complexity-SVRG Catalyst Nonlinear} is an extension of Corollary 3.7 in \cite{li2022new} to nonlinear constraints, showing that the penalty reformulation results are also applicable to problems with nonlinear constraints. Corollary 2 of \cite{mishchenko2018stochastic} establishes an \(O(\frac{1}{t^2})\) convergence rate for the objective value gap (with an additional projection) and \(O(\frac{1}{t^2})\) for the distance between the iterate to the feasible region in terms of the number of iterations \(t\) in the long run. Our Corollary \ref{cor:Static Complexity-SVRG Catalyst Nonlinear} establishes \(\tilde O(\frac{1}{t^2})\) and \(\tilde O(\frac{1}{t^2})\) convergence rates for the objective value gap and violations of constraints, which is similar to their results. Despite slight differences of the settings, we want to emphasize the following advantages of our results compared with theirs. First, our result has a better dependence on the Lipschitz constant of the objective and the initial condition. Second, their results are based on a global Hoffman-type assumption (Assumption 1), and their convergence rates depend on the corresponding Hoffman constant (\(\gamma\)), which are difficult to estimate and bound in practice, whereas our results only requires a bound on the norm of the dual solution. 
{Third, Theorem~\ref{thm:estimation nonlinear new z infty m 01} allows us to generalize other results for linear constraints, including the nested version of the penalty method (Algorithm 3.1 and Proposition 3.3), the construction of the dual solutions (Proposition 4.1), and the screening procedure (Algorithm 5.1 and Proposition 5.5) of \cite{li2022new}, to problems with nonlinear constraints. We leave these generalizations as future work.}
 
\subsection{Approximation of Non-Smooth Constraints}
\label{subsec:approx constraints}
In this subsection, we address the situation where the constraints \(a_i\) are non-smooth, and show how to apply our penalty reformulation to a suitably chosen smooth approximation. For example, for SDP problems, the constraint 
\(M(x)=M_0+\sum_{i=1}^n x_n M_n\preceq 0\), where \(M_i\in S^{k}\) are symmetric matrices, can be considered as a non-smooth nonlinear constraint \(\lambda_{\max}(M(x))\le 0\). One possible way to address such non-smooth problems is to form smooth approximations of the non-smooth constraints and solve a related smoothly constrained problem (e.g. for SDP problems consider Nesterov's smoothing technique \citet{nesterov2005smooth}). We consider applying our penalty reformulation, and associated gradient methods as developed in the previous section, to the resulting smoothly constrained problem. In particular, given smooth approximations \(\hat{a}_i\) of the non-smooth \(a_i\) functions, consider the two penalty reformulations: 
\[\begin{aligned}
F_{\xi,\delta}(x)& = F(x)+\xi\sum_{i=1}^mp_{\delta}(a_i(x)),\\
\hat F_{\xi,\delta}(x)& = F(x)+\xi\sum_{i=1}^mp_{\delta}(\hat a_i(x)),\\
\end{aligned}\] with optimal solutions \(x_{\xi,\delta}^*\) and \(\hat x_{\xi,\delta}^*\), respectively. %
\begin{remark}\label{prop:approx constraint}
Suppose the errors between \(a_i\) and \(\hat a_i\) are bounded as
\(|\hat a_i(\hat x_{\xi,\delta}^*)-a_i(\hat x_{\xi,\delta}^*)|\le \Delta\) and \(|\hat a_i(x^*)-a_i(x^*)|\le \Delta\), and we have an approximate solution \(\tilde x\), s.t. \(\hat F_{\xi,\delta}(\tilde x)-\hat F_{\xi,\delta}(\hat x_{\xi,\delta}^*)\le\xi\varepsilon\), {then \(\tilde x\) is an \(O(m\Delta+m\delta+\xi\epsilon)\)-approximate solution}. Since \[\begin{aligned}
&F(x^*)=F_{\xi,0}(x^*)\le F_{\xi,0}(\tilde x)\le F_{\xi,\delta}(\tilde x)\\ &\le \hat F_{\xi,\delta}(\tilde x) + m\xi\Delta\le  \hat F_{\xi,\delta}(\hat x_{\xi,\delta}^*)+m\xi\Delta +\xi\varepsilon\\ &\le \hat F_{\xi,\delta}(x^*)+m\xi\Delta+\xi\varepsilon\le F(x^*)+2m\xi\Delta+m\xi\delta\log2+\xi\varepsilon,    
\end{aligned}
\] when \(\xi\ge 2\bar \xi\), we have 
\[\begin{aligned}
F(\tilde x)-F(x^*)&\le\epsilon_1= m\xi\Delta+m\xi\delta\log2+\xi\varepsilon,\\ \|\bar A(\tilde x)\|_1&\le \epsilon_2=4m\Delta+2m\delta\log2+2\xi\varepsilon.
\end{aligned}\]
The second estimation comes from Proposition~\ref{prop:nonstrongly convex estimation}. Though we state Proposition~\ref{prop:nonstrongly convex estimation} for smooth constraints for consistency, it can be adapted for non-smooth constraints, as an application of convexity and KKT condition. In other words, \(\tilde x\) is an \(\left(\epsilon_1,\epsilon_2\right)\)-approximate solution with respect to the $1$-norm.
\end{remark}

Remark \ref{prop:approx constraint} enables us to solve the original problem \eqref{eq:Original Nonlinear} by solving the penalty reformulations of problems with smoothing constraints.

\section{Conclusion}
We construct penalty reformulations for strongly convex function minimization subject to nonlinear constraints and provide upper bounds on the objective value gap between the constrained problem and the penalized problem and on the violation of constraints. 
{These bounds lead to a penalty algorithm with
\(\tilde O\left(m^{\frac54}/\sqrt{\mu\epsilon}\right)\) 
expected stochastic gradient iterations to obtain an \(\epsilon\)-approximate solution for a strongly convex problem with $m$ constraints.}
Further applications, including problems with non-smooth constraints are also analyzed.
Using the bounds provided, further extensions of the penalty algorithm, including a nested version and the application of the screening procedure, can be incorporated into the algorithm, to achieve a better complexity in theory and practice, which we leave as future directions. %

\printcredits

\bibliographystyle{cas-model2-names}

\bibliography{Newbib}

\end{document}